\documentclass[13pt]{article}  
\usepackage{amssymb}              
\usepackage{amsthm}
\usepackage{eucal}
\usepackage{verbatim}
\usepackage[dvips]{graphicx}
\usepackage{multirow}
\usepackage{fancyhdr}
\usepackage{color}
\usepackage{enumerate}
\usepackage{calrsfs}
\usepackage{fullpage}
\usepackage{amsmath}
\usepackage[pagewise]{lineno}

\newtheorem{theorem}{Theorem}
\newtheorem{lemma}{Lemma}
\newtheorem{definition}{Definition}
\newtheorem{remark}{Remark}

\newtheorem{corollary}{Corollary}

\makeatletter
\newcommand{\leqnomode}{\tagsleft@true}
\newcommand{\reqnomode}{\tagsleft@false}
\makeatother

\def\({\begin{eqnarray}}
\def\){\end{eqnarray}}
\def\[{\begin{eqnarray*}}
\def\]{\end{eqnarray*}}
\def\part#1#2{\frac{\partial #1}{\partial #2}}

\def\R{\mathbb{R}}
\def\d{\mathrm{d}}
\def\tot#1#2{\frac{\d #1}{\d #2}}
\def\eps{\varepsilon}
\def\wpsi{\widetilde{\psi}}
\def\wv{\widetilde{v}}
\def\L{\mathcal{L}}

\def\rev#1{\textcolor{red}{#1}}

\begin{document}

\title{Asymptotic flocking in the
Cucker-Smale model with reaction-type delays
in the non-oscillatory regime}   
\author{Jan Haskovec \qquad Ioannis Markou}         
\date{February 2020}  
\maketitle

\begin{abstract}
We study a variant of the Cucker-Smale system with reaction-type delay.
Using novel backward-forward and stability estimates
on appropriate quantities
we derive sufficient conditions
for asymptotic flocking of the solutions.
These conditions, although not explicit,
relate the velocity fluctuation of the initial datum
and the length of the delay.
If satisfied, they guarantee monotone decay (i.e., non-oscillatory regime)
of the velocity fluctuations towards zero for large times.
For the simplified setting with only two agents
and constant communication rate
the Cucker-Smale system reduces to the delay negative
feedback equation. We demonstrate that in this case our method
provides the sharp condition for the size of the delay
such that the solution be non-oscillatory.
Moreover, we comment on the mathematical issues
appearing in the formal macroscopic description of the
reaction-type delay system.
\end{abstract}
\vspace{2mm}

\textbf{Keywords}: Cucker-Smale system, flocking, time delay, velocity fluctuation.
\vspace{2mm}

\textbf{2010 MR Subject Classification}: 34K05, 82C22, 34D05, 92D50.
\vspace{2mm}

\section{Introduction}\label{sec:Intro}
The study of collective behavior of autonomous self-propelled agents has
attracted significant interest in various scientific disciplines,
such as biology, sociology, robotics, economics etc. The main
motivation is to model and explain the emergence of
self-organized patterns on the global scale, while individual agents
typically interact only locally. From the vast amount of literature
on mathematical theory of collective phenomena in biology, social sciences
and engineering we refer to the
relatively recent surveys \cite{Naldi-Pareschi-Toscani,
Pareschi-Toscani-survey, Vicsek-survey, ChHaLi} and the references
therein. The newest developments are captured in, e.g.,
\cite{Ma,PiTr, PiVa18, PiVa, Ha1, Kalise, DHK19}.

The Cucker-Smale model is a
prototypical model of consensus seeking,
or, in physical context, velocity alignment.
Introduced in  \cite{CuSm1, CuSm2}, it has been extensively
studied in many variants, where the main point of interest
is the asymptotic convergence of the (generalized) velocities
towards a consensus value.
In this paper we focus on a variant of the Cucker-Smale model with reaction-type delays
with a fixed delay $\tau>0$.
We consider $N \in\mathbb{N}$ autonomous agents described by their phase-space
coordinates $(x_i(t), v_i(t))\in\mathbb{R}^{2d}$, $i=1,2,\cdots,N$,
$t\geq 0$, where $x_i(t) \in \mathbb{R}^{d}$, resp. $v_i(t) \in
\mathbb{R}^{d}$, are time-dependent position, resp. velocity,
vectors of the $i$-th agent,
and $d\geq 1$ is the physical space dimension.
For notational convenience, we shall suppress the explicit time
dependence of the phase space coordinates in the sequel, i.e. we shall
write $x_{i}$ instead of $x_{i}(t)$ and $v_{i}$ instead of
$v_{i}(t)$. Moreover, we use the symbol \, $\widetilde{}$ \,  to
denote the value of a variable at time $t-\tau$, i.e.
$\widetilde{x}_{i}:=x_{i}(t-\tau)$ and
$\widetilde{v}_{i}:=v_{i}(t-\tau)$, where $\tau>0$ is a fixed time delay.

We shall study the following Cucker-Smale type system
\(
      \dot{x}_{i} &=& v_{i} \label{CS_delay1} \\
      \dot{v}_{i} &=& \frac{\lambda}{N}\sum\limits_{j=1}^N
      \psi(|\widetilde{x}_{i}-\widetilde{x}_{j}|)
      (\widetilde{v}_{j}-\widetilde{v}_{i}), \label{CS_delay2}
\)
for $i=1,2,\cdots,N$, with the parameter $\lambda>0$.
As initial data we prescribe the initial
position and velocity trajectories for $i=1,2,\cdots,N$,
\( \label{CS_delay_IC}
     (x_{i}(t),v_{i}(t)) \equiv (x_{i}^0(t),v_{i}^0(t)) \qquad \mbox{for } t \in [-\tau,0],
\)
with $(x_{i}^0(s), v_{i}^0(s)) \in C([-\tau,0];\mathbb{R}^{2d})\cap C^1((-\tau,0);\mathbb{R}^{2d})$.
A generic case is given by constant initial trajectories
 $(x_i^0 ,v_i^0) \in\mathbb{R}^{2d}$, $i=1,2,\cdots,N$.
The function $\psi: [0,\infty) \to (0,\infty)$ is a
positive nonincreasing differentiable function that models the communication rate
between two agents $i$, $j$, in dependence of their 
metric distance. For notational convenience, we shall denote
\[
   \psi_{ij}:=\psi(|x_{i}-x_{j}|) \qquad\mbox{and}\qquad
   \widetilde{\psi}_{ij}:=\psi(|\widetilde{x}_{i}-\widetilde{x}_{j}|),
\]
where $|\cdot|$ denotes the Euclidean distance in $\mathbb{R}^d$.
In our paper we shall introduce the following three assumptions on $\psi=\psi(r)$,
namely, that
\( \label{ass:psi0}
   \psi(r) \leq 1 \qquad\mbox{for all } r\geq 0,
   \)
that there exist some $\gamma<1$ and $c, R>0$ such that
\( \label{ass:psi1}
   \psi(r) \geq c r^{-1+\gamma} \qquad\mbox{for all } r\geq R,
\)
and that there exists $\alpha>0$ such that
\( \label{ass:psi2}
   \psi'(r) \geq -\alpha \psi(r) \qquad\mbox{for all } r > 0.
\)
The prototype rate considered by Cucker and Smale in \cite{CuSm1, CuSm2}
and many subsequent papers is of the form
\begin{equation}
\label{CS_cw}
   \psi(r)=\frac{1}{(1+r^2)^{\beta}},
\end{equation}
with the exponent $\beta\geq 0$.
The assumption \eqref{ass:psi1} is verified for \eqref{CS_cw} if $\beta<1/2$,
while assumption \eqref{ass:psi2} is satisfied for all $\beta\geq 0$ by choosing $\alpha:=2\beta$.
Let us point out that the results of our paper are not restricted to the particular form \eqref{CS_cw}
of the communication rate.


The system \eqref{CS_delay1}--\eqref{CS_delay2} represents a model
for flocking or consensus dynamics, where the agents react to the
information perceived from their surroundings with a given
\emph{processing or reaction delay} $\tau>0$, assumed to be constant for all
agents. This is in contrast to the \emph{transmission delay}, caused
by finite propagation speed of information, which would induce
delays depending on the distance between agents (i.e.,
state-dependent delay). In the model
\eqref{CS_delay1}--\eqref{CS_delay2} we assume instantaneous
propagation of information between agents, which is certainly
justified for visual communication in animal groups or radio
communication in groups of terrestrial vehicles or robots. On the
other hand, reaction times for both living agents and artificial robots
are positive and might have a notable effect on the collective behavior of the system.
Therefore, we consider them as the only source of delay in our system.

The main objective in the study of Cucker-Smale type models is their
asymptotic behavior, in particular, the concept of \emph{conditional
or unconditional flocking}. In agreement with \cite{CuSm1, CuSm2}
and many subsequent papers,
we say that the system exhibits flocking behavior if there is
asymptotic alignment of velocities and the particle group stays
uniformly bounded in time.

\begin{definition} \label{def:flocking}
We say that the particle system \eqref{CS_delay1}--\eqref{CS_delay2}
exhibits \emph{flocking} if its solution $(x(t),v(t))$ satisfies
\begin{equation*}
   \lim_{t \to \infty} |v_i-v_j| =0, \qquad
   \sup_{t \geq 0} |x_i-x_j| < \infty,
\end{equation*}
for all $i,j = 1,2,\cdots,N$.
\end{definition}

The term \emph{unconditional flocking} refers to
the case when flocking behavior takes place for all initial
conditions, independently of the value of the parameters $\lambda>0$
and $N\in\mathbb{N}$. The celebrated result of Cucker and Smale
\cite{CuSm1, CuSm2} states that the system
\eqref{CS_delay1}--\eqref{CS_delay2} without delay (i.e., $\tau=0$)
with the communication rate \eqref{CS_cw} exhibits unconditional
flocking if and only if $\beta < 1/2$. For $\beta \geq 1/2$ the
asymptotic behavior depends on the initial configuration and the
particular value of the parameters $\lambda>0$ and $N\in\mathbb{N}$.
In this case we speak about \emph{conditional flocking}. The proof
of Cucker and Smale (and its subsequent variants, see \cite{HaTa,
CaFoRoTo}) is based on a bootstrapping argument, estimating, in
turn, the quadratic fluctuations of positions and velocities,
and showing that the velocity fluctuations decay
monotonically to zero as $t\to\infty$.

Let us point out that the presence of delays,
like in our system \eqref{CS_delay1}--\eqref{CS_delay2},
may lead to oscillations of velocity fluctuations, i.e.,
they may be no more monotonically decreasing in time.
Monotone decay can only be expected for small enough values
of the delay $\tau>0$ and particular initial data.
The main task of this paper is to develop analytical methods
for the non-oscillatory regime of the velocity fluctuations.
In particular, we shall prove that for each initial datum $(x^0(s),v^0(s))$
satisfying a certain condition,
there exists a critical delay $\tau_c>0$ such that the solution of the system
\eqref{CS_delay1}--\eqref{CS_delay2} exhibits flocking in the sense
of Definition~\ref{def:flocking} whenever $\tau<\tau_c$.
Moreover, the quadratic velocity fluctuations decay monotonically
to zero as $t\to\infty$ with exponential rate.
The proof shall be based on novel type estimates
of the local growth of the functional $D=D(t)$ defined below \eqref{def:D}.
The core idea of our method will be explained in Section \ref{sec:simple}
where we consider the simplified case $N=2$, $\psi\equiv 1$
leading to the delay negative feedback equation.
We shall show that our method provides sharp value
of the critical time delay for the non-oscillatory regime
of delay negative feedback.

Flocking in Cucker-Smale type models with delay and renormalized
communication weights was recently studied in \cite{LiWu, ChHa1}.
Both these papers consider the case where the delay in the
velocity equation for the $i$-th agent is present only in the
$v_j$-terms for $j\neq i$.
This allows for using convexity arguments
to conclude a-priori uniform boundedness of the velocities.
Such convexity arguments are not applicable for our
system \eqref{CS_delay1}--\eqref{CS_delay2}.
In \cite{ChHa2} the method is extended to the mean-field limit ($N\to\infty$) of the model.
In \cite{ChLi} the authors consider heterogeneous delays both in the
$x_j$ and $v_j$ terms and they prove asymptotic flocking for small
delays with the weights \eqref{CS_cw}. A system with time-varying
delays was studied in \cite{PiTr}, under the a-priori assumption
that the Fiedler number (smallest positive eigenvalue) of the
communication matrix $(\psi_{ij})_{i,j=1}^N$ is uniformly bounded
away from zero. The same assumption is made in \cite{ErHaSu}
for a Cucker-Smale type system with delay and multiplicative noise.
The validity of this relatively strong assumption
would typically be guaranteed by making the communication rates
$\psi_{ij}$ a-priori bounded away from zero, which excludes the
generic choice \eqref{CS_cw} for $\psi$. The main advantage of our
method is that we do not require such a-priori boundedness.

Let us now introduce for $t\geq-\tau$ the quadratic fluctuation of the velocities
\(  \label{def:V}
   V(t):=\frac{1}{2}\sum_{i=1}^N \sum_{j=1}^N |v_{i}-v_{j}|^{2}
\)
and the quantity
\(  \label{def:D}
   D(t):=\frac{1}{2} \sum_{i=1}^N \sum_{j=1}^N {\psi}_{ij} |{v}_{j} -{v}_{i}|^{2}.
\)
Moreover, let us define the quantity
\( \label{def:L0}
      L^0 := (2\lambda\tau + 1) e^{2\lambda\tau} \max_{\vartheta\in [-\tau,0]} V(\vartheta) + 4\tau\lambda^3 \int_{-\tau}^{0}\int_{\theta}^{0}  {D}(s) \, \d s \, \d\theta,
\)
which is calculated solely from the initial datum $(x^0(s), v^0(s))$.
Note that $L^0=V(0)$ for $\tau=0$.

Our main result is the following:

\begin{theorem} \label{thm:main}
Let the communication rate $\psi=\psi(r)$ verify the assumptions \eqref{ass:psi0}--\eqref{ass:psi2}.
Let the initial datum $(x^0(s),v^0(s))$ be such that
\(   \label{ass:main1}
   M^0 := \max \left\{\sup_{s \in (-\tau,0)} \frac{|\dot{D}(s)|}{D(s)},
                 \frac{|\dot{D}(0^+)|}{D(0)}
              \right\} < \infty,
\)
with $D$ given by \eqref{def:D} and $\dot D(0+)$ denoting the right derivative of $D$ at $t=0$.
Then, there exists a critical delay $\tau_c>0$,
depending on the initial datum through $L^0$ and $M^0$
and calculable as a solution of a system of two nonlinear algebraic equations,
such that if $\tau<\tau_c$,
the solution of the system \eqref{CS_delay1}--\eqref{CS_delay2}
subject to initial datum \eqref{CS_delay_IC}
exhibits flocking behavior in the sense of Definition \ref{def:flocking}.
Moreover, the quadratic velocity fluctuation $V=V(t)$ decays monotonically
and exponentially to zero as $t\to\infty$.
\end{theorem}

The value of the critical delay $\tau_c$ cannot be given explicitly,
as it is a solution of a system of two highly nonlinear algebraic equations.
However, the proof of Theorem \ref{thm:main} is constructive in the sense
that it gives a detailed recipe how $\tau_c$ can be found given the values
of $\lambda$, $L^0$ and $M^0$. Here we merely remark that $\tau_c$
decreases with increasing values of these three quantities.
Moreover, let us point out that the value of $L^0$ explicitly depends
on the velocity fluctuations of the initial datum. This is corresponds to the
statistical mechanics intuition about synchronization systems - namely,
that asymptotic consensus may not be reached if initially the system is
in a strongly disordered state. The natural measure of order in this context is the velocity
fluctuation. We refer to \cite{Kuramoto} where a rigorous connection between the
Cucker-Smale flocking model and the Kuramoto synchronization model was established.
Moreover, let us note that the velocity fluctuation of the initial datum also appears
in the sufficient condition for conditional flocking derived in the original papers
by Cucker and Smale \cite{CuSm1, CuSm2} as well as in \cite{HaLiu}.

In the generic case of constant initial datum \eqref{CS_delay_IC},
condition \eqref{ass:main1} simplifies to
\[
   M^0 := \frac{|\dot{D}(0^+)|}{D(0)} < \infty,
\]
and, as we shall demonstrate in Remark \ref{rem:constantIC}, it is implicitly verified.
We can thus formulate the following:

\begin{corollary}
Let the communication rate $\psi=\psi(r)$ verify the assumptions \eqref{ass:psi0}--\eqref{ass:psi2}
and let the initial datum $(x^0,v^0)$ be constant on $[-\tau,0]$.
Then, there exists a critical delay $\tau_c>0$
such that if $\tau<\tau_c$,
the solution of the system \eqref{CS_delay1}--\eqref{CS_delay2}
subject to initial datum \eqref{CS_delay_IC}
exhibits flocking behavior in the sense of Definition \ref{def:flocking}.
Moreover, the quadratic velocity fluctuation $V=V(t)$ decays monotonically
and exponentially to zero as $t\to\infty$.
\end{corollary}

The paper is organized as follows. In Section \ref{sec:simple}
we consider the simplified case $N=2$, $\psi\equiv 1$ to demonstrate
the main idea of our approach without having to tackle the technicalities
of the general case. In Section \ref{sec:flocking} we apply the method
to prove the existence of a critical delay for asymptotic flocking in
the Cucker-Smale type system \eqref{CS_delay1}--\eqref{CS_delay2}.
Finally, in Section \ref{sec:meanfield} we make a short remark
about its formal mean-field limit and the complications cause
by the presence of reaction-type delay.

\section{The simplified case - delay negative feedback}\label{sec:simple}
To gain insight into the main idea of our analytical approach,
let us consider the maximally simplified version of \eqref{CS_delay1}--\eqref{CS_delay2},
namely, the setting with two agents only, $N=2$, living on a line, $d=1$,
and constant communication rate $\psi \equiv 1$.
Then, \eqref{CS_delay1} decouples from \eqref{CS_delay2}
and setting $u(t) := v_1(t) - v_2(t)$, we obtain from \eqref{CS_delay2}
the \emph{delay negative feedback equation}
\( \label{feedback}
   \dot u(t) = -\lambda u(t-\tau),
\)
with $\lambda>0$.
For simplicity, we shall consider the constant initial datum $0\neq u^0 \in\R$,
\(  \label{feedback_IC}
   u(t) \equiv u^0\qquad \mbox{for } t\in [-\tau,0].
\)
Delay negative feedback is arguably the simplest nontrivial delay differential equation.
Despite its simplicity, it exhibits a surprisingly rich qualitative dynamics.
Rescaling of time $t \mapsto t/\tau$ in \eqref{feedback} leads to
\[
    \dot u(t) = -\lambda\tau u(t-1),
\]
while the rescaling $t\mapsto \lambda t$ gives
\[
    \dot u(t) = - u(t-\lambda\tau).
\]
Consequently, the dynamics of \eqref{feedback} depends only on the single parameter $\lambda\tau > 0$.
However, we prefer to keep \eqref{feedback} in its non-rescaled form in order to easily relate the analysis
performed in this section to the rest of the paper.
An analysis of the corresponding characteristic equation
\[
   z + \lambda\tau e^{-z} = 0,
\]
where $z\in\mathbb{C}$, reveals that:
\begin{itemize}
\item
If $0 < \lambda\tau < e^{-1}$, then $u=0$ is asymptotically stable.
\item
If $e^{-1} < \lambda\tau < \pi/2$, then $u=0$ is asymptotically stable,
but every nontrivial solution of \eqref{feedback} is oscillatory.
\item
If $\lambda\tau > \pi/2$, then $u=0$ is unstable.
\end{itemize}
In fact, if $\lambda\tau < e^{-1}$, the solutions subject to the constant initial datum
never oscillate and tend to zero as $t\to\infty$.
If $\lambda\tau$ becomes larger than $e^{-1}$ but
smaller than $\pi/2$, the nontrivial solutions must oscillate
(i.e., change sign infinitely many times as $t\to\infty$),
but the oscillations are damped and vanish as $t\to\infty$.
Finally, for $\lambda\tau > \pi/2$ the nontrivial solutions oscillate with unbounded amplitude
as $t\to\infty$.
We refer to Chapter 2 of \cite{Smith} and \cite{Gyori-Ladas} for details.

The traditional tool in study of asymptotic behavior of Cucker-Smale type model
is the quadratic velocity fluctuation. In our present setting with two agents only,
the quadratic velocity fluctuation is simply $y(t):=u(t)^2/2$.
Taking the time derivative, we readily have
\(  \label{doty}
   \dot y = -\lambda u \widetilde u,
\)
and, obviously, we only have a chance to prove that $y=y(t)$ is
decaying for all $t>0$ if the solution $u=u(t)$ does not oscillate,
which requires $\lambda\tau < e^{-1}$ (and eventually an appropriate
assumption on the initial datum in case it was not constant).
The core of the method developed in this paper is to prove that,
under these conditions, the quadratic fluctuation $y=y(t)$
is indeed (exponentially) decaying to zero as $t\to\infty$.
The method consists of two steps.

\textbf{Step 1.}
We start by applying the Cauchy-Schwartz inequality  in \eqref{doty}, with some $\eps>0$
to be specified later,
\(  \label{eps-est}
   | \dot y| \leq \frac{\lambda\eps}{2} |u|^2 + \frac{\lambda}{2\eps} |\widetilde u|^2
      = \lambda\eps y + \frac{\lambda}{\eps} \widetilde y.
\)
Let us observe that, since the initial datum $u^0$ is nonzero and constant, we have
\[
   -\frac{\dot y(0+)}{y(0)} = 2\lambda \frac{u(0) u(-\tau)}{u(0)^2} = 2\lambda,
\]
where $\dot y(0+)$ is the right-sided derivative of $y$ at zero
(observe that $\dot y(0)$ does not exist).
Let us now choose some $\mu>2$, then by continuity
of $\dot y/y$ we have
\[   
   -\frac{\dot y(t)}{y(t)} < \mu\lambda \qquad\mbox{for } t\in (0,T)
\]
for some $T>0$. Integration on the interval $(t-s,t) \subseteq (0,T)$ gives
\(   \label{flow_est2}
   y(t-s) \leq e^{\mu\lambda s} y(t),
\)
and due to the constant initial datum, we can extend the validity of \eqref{flow_est2}
for all $-\tau \leq t-s < t < T$.
We shall now argue that if $\lambda\tau < e^{-1}$ then
there exists $\mu>2$ such that \eqref{flow_est2}
holds globally, i.e., $T=+\infty$.

For contradiction, let us assume that $T<+\infty$.
Then, by continuity, we have
\(  \label{contT}
   -\frac{\dot y(T)}{y(T)} = \mu\lambda.
\)
However, using \eqref{flow_est2} with $t:=T$, $s:=\tau$ in \eqref{eps-est} gives
\[
   | \dot y(T)| \leq \lambda\eps y(T) + \frac{\lambda}{\eps} y(T-\tau) \leq
      \lambda y(T) \left( \eps + \frac{1}{\eps} e^{\mu\lambda\tau} \right).
\]
Minimizing the right-hand side with respect to $\eps>0$ gives
\[
   | \dot y(T)| \leq 2 \lambda e^\frac{\mu\lambda\tau}{2} y(T).
\]
A simple calculation reveals that if (and only if) $\lambda\tau < e^{-1}$, there exists $\mu >2$
such that $2e^\frac{\mu\lambda\tau}{2} < \mu$.
In particular, the choice $\mu:=2e$ works for all $\lambda\tau < e^{-1}$.
But then
\[
   | \dot y(T)| \leq 2 \lambda e^\frac{\mu\lambda\tau}{2} y(T) < \mu\lambda y(T),
\]
which is a contradiction to \eqref{contT} and we conclude that $T=+\infty$.
Consequently, we have the estimate
\(   \label{flow_est3}
   y(t-s) \leq e^{2e \lambda s} y(t)\qquad\mbox {for all} -\tau < t-s < t < +\infty.
\)

\textbf{Step 2.}
We have for $t>0$
\(  \label{for difference}
   \dot y = - \lambda u\widetilde u = -\lambda (u-\widetilde u)\widetilde u - \lambda \widetilde u^2
     \leq \lambda |u-\widetilde u| |\widetilde u| - \lambda \widetilde u^2,
\)
and, restricting to $t>\tau$,
\(   \label{u-wtu}
   |u-\widetilde u| \leq \int_{t-\tau}^t |\dot u(s)| \d s \leq \lambda \int_{t-\tau}^t |u(s-\tau)| \d s = \lambda\sqrt{2} \int_{-\tau}^0 \sqrt{y(t-\tau+s)} \d s,
\)
where we used the definition $y = u^2/2$.
>From \eqref{flow_est2} we obtain $y(t-\tau+s) \leq e^{-2e\lambda s} y(t-\tau)$ for $s\in(-\tau,0)$, which gives
\[
   |u-\widetilde u| \leq \lambda\sqrt{2\widetilde y} \int_{0}^\tau e^{e\lambda s} \d s
      \leq  \lambda\tau \sqrt{2\widetilde y} e^{e\lambda\tau}.
\]
Therefore, for $t>\tau$,
\(   \label{almost there}
   \dot y \leq 2\lambda \left( \lambda\tau e^{e\lambda\tau} - 1 \right) \widetilde y. 
\)
Consequently, for $\lambda\tau < e^{-1}$ we readily have $\dot y \leq 0$ for all $t>\tau$,
so that $y=y(t)$ has a nonnegative limit as $t\to \infty$.
If this limit was strictly positive, then the limit of the right-hand side in \eqref{almost there}
would be strictly negative, which would imply $\limsup_{t\to\infty} \dot y(t) < 0$, a contradiction.
We conclude that, if $\lambda\tau < e^{-1}$,
\[
   \lim_{t\to\infty} y(t) = 0.
\]

Let us remark that at the cost of reducing the range of admissible values of $\lambda\tau$
we are able to prove the exponential decay of $y=y(t)$ to zero.
We start by writing
\[
   \dot y = - \lambda u\widetilde u = -\lambda u(\widetilde u - u) - \lambda u^2
     \leq \lambda |u-\widetilde u| |u| - \lambda u^2,
\]
for $t>0$; notice the difference to \eqref{for difference}.
We again combine \eqref{u-wtu} with \eqref{flow_est2}, but this time we take $y(t-\tau+s) \leq e^{2e\lambda (\tau-s)} y(t)$ for $s\in(-\tau,0)$, which gives
\[
   |u-\widetilde u| \leq \lambda\sqrt{2y} \int_{\tau}^{2\tau} e^{e\lambda s} \d s
      \leq  \lambda\tau \sqrt{2 y} e^{2e\lambda\tau}.
\]
Consequently, for $t>\tau$,
\[
   \dot y \leq 2\lambda \left(\lambda\tau e^{2e\lambda\tau} - 1 \right) y .
\]
This implies exponential convergence of $y=y(t)$ to zero if $0 \leq \lambda\tau < z^\ast$
where $z^\ast \simeq 0.252$ is the real root of the transcendental equation $z e^{2ez} - 1 = 0$.
Clearly, $0 < z^\ast < e^{-1}$.

To summarize, we proved that if $\lambda\tau < e^{-1}$,
the solution $u=u(t)$ of \eqref{feedback} subject to the constant initial datum
\eqref{feedback_IC} converges monotonically (i.e., non-oscillatory) towards zero as $t\to\infty$.
Let us point out that this result is sharp since we know that if $\lambda\tau > e^{-1}$,
oscillations appear.
Moreover, if $\lambda\tau$ is smaller than approx. $0.252$, the convergence
is exponential.
The goal of this paper is to extend the above method to the full Cucker-Smale system
\eqref{CS_delay1}--\eqref{CS_delay2}.
The core idea is to derive an analogue of the forward-backward estimate \eqref{flow_est3},
however, due to the coupling of the velocity and position variables in the full Cucker-Smale model,
the forward-backward estimate has to be applied to the quantity $D=D(t)$ given by \eqref{def:D}
instead of the quadratic velocity fluctuation.

\section{Asymptotic flocking for the delay Cucker-Smale model}\label{sec:flocking}

The proof of asymptotic flocking for the system \eqref{CS_delay1}--\eqref{CS_delay2}
will be carried out in three steps: First, in Section \ref{subsec:VflucB}
we shall derive an uniform bound on the quadratic velocity fluctuation
$V=V(t)$ by constructing a suitable Lyapunov functional.
Then, in Section \ref{FbEst} we prove a forward-backward estimate
on the quantity $D=D(t)$ defined in \eqref{def:D}, which states
that $D=D(t)$ changes at most exponentially locally in time.
Finally, in Section \ref{subsec:EstV} we prove the asymptotic decay
of the quadratic velocity fluctuation and boundedness of the spatial
fluctuation and so conclude the proof of Theorem \ref{thm:main}.

We start by noting that unique solutions of the system \eqref{CS_delay1}--\eqref{CS_delay_IC}
are constructed by the method of steps, see, e.g., \cite{Smith}.
We also note that the symmetry of the particle interactions $\psi_{ij} = \psi_{ji}$
implies that the total momentum is conserved along the solutions of
\eqref{CS_delay2}, i.e.,
\(  \label{mom_cons}
   \sum_{i=1}^N v_i(t) = \sum_{i=1}^N v_i(0) \qquad\mbox{for all } t\geq 0.
\)

\subsection{Uniform bound on the velocity fluctuations}\label{subsec:VflucB}

We first derive an estimate on the dissipation of the quadratic velocity fluctuation
in terms of the quantity $D=D(t)$ defined in \eqref{def:D}.

\begin{lemma} \label{lem:dV}
For any $\delta>0$ we have,
along the solutions of \eqref{CS_delay1}--\eqref{CS_delay2},
\begin{align} \label{dVest}
   \tot{}{t} V(t) \leq 2(\delta - 1)\lambda \widetilde{D}(t)
      + \frac{2\tau \lambda^3}{\delta} \int_{t-\tau}^{t} \widetilde{D}(s) \,\d s
      \qquad\mbox{for all } t>\tau.
\end{align}
\end{lemma}

\begin{proof}
We have
\[
   \tot{}{t} V(t)
   =
   \sum_{i=1}^N\sum_{j=1}^N \langle v_i-v_j, \dot v_i - \dot v_j \rangle
   =
   2N \sum_{i=1}^N \langle v_i, \dot v_i \rangle -
   \sum_{i=1}^N\sum_{j=1}^N \langle v_i, \dot v_j \rangle
   =
   2N \sum_{i=1}^N \langle v_i, \dot v_i \rangle,
\]
where the last equality is due to the conservation of momentum \eqref{mom_cons}.
With \eqref{CS_delay2} we have
\begin{align*}
   \nonumber
   \tot{}{t} V(t) &=
   2\lambda \sum_{i=1}^N \sum_{j=1}^N \langle v_{i}, \widetilde{\psi}_{ij}(\widetilde{v}_{j}-\widetilde{v}_{i})\rangle \\
   \nonumber
   &=
    2\lambda \sum_{i=1}^N \sum_{j=1}^N \langle
   \widetilde{v}_{i},\widetilde{\psi}_{ij}(\widetilde{v}_{j}-\widetilde{v}_{i})\rangle -
   2\lambda \sum_{i=1}^N \sum_{j=1}^N \langle
   \widetilde{v}_{i}-v_{i},\widetilde{\psi}_{ij}(\widetilde{v}_{j}-\widetilde{v}_{i})\rangle.
\end{align*}
For the first term of the right-hand side we apply the standard
symmetrization trick (exchange of summation indices $i
\leftrightarrow j$, noting the symmetry of $\widetilde{\psi}_{ij} =
\widetilde{\psi}_{ji}$),
\[
   2\lambda
   \sum_{i=1}^N \sum_{j=1}^N \langle
   \widetilde{v}_{i},\widetilde{\psi}_{ij}(\widetilde{v}_{j}-\widetilde{v}_{i})\rangle = -
   \lambda \sum_{i=1}^N \sum_{j=1}^N \widetilde{\psi}_{ij} |\widetilde{v}_{j} -\widetilde{v}_{i}|^{2}.
\]
Therefore, we arrive at
\begin{align*}
   \tot{}{t} V(t) = -
   \lambda \sum_{i=1}^N \sum_{j=1}^N \widetilde{\psi}_{ij} |\widetilde{v}_{j} -\widetilde{v}_{i}|^{2} -
   2\lambda \sum_{i=1}^N \sum_{j=1}^N \langle \widetilde{v}_{i}-
   v_{i},\widetilde{\psi}_{ij}(\widetilde{v}_{j}-\widetilde{v}_{i})\rangle.
\end{align*}
For the last term we use the Young inequality with $\delta>0$ and assumption \eqref{ass:psi0},
\begin{equation*}
   \left|
      2\lambda \sum_{i=1}^N \sum_{j=1}^N \langle \widetilde{v}_{i}- v_{i},\widetilde{\psi}_{ij}(\widetilde{v}_{j}-\widetilde{v}_{i})\rangle
   \right|
     \leq
   \lambda\delta \sum_{i=1}^N \sum_{j=1}^N \widetilde{\psi}_{ij} |\widetilde{v}_{j} -\widetilde{v}_{i}|^{2} +
   \frac{N\lambda}{\delta} \sum_{i=1}^N  |\widetilde{v}_{i}-v_{i}|^{2}.
\end{equation*}
Hence,
\begin{align}
   \nonumber
   \tot{}{t}V(t) &\leq
   (\delta-1)\lambda
      \sum_{i=1}^N \sum_{j=1}^N \widetilde{\psi}_{ij} |\widetilde{v}_{j} -\widetilde{v}_{i}|^{2} +
    \frac{N\lambda}{\delta} \sum_{i=1}^N |\widetilde{v}_{i}-v_{i}|^{2} \\
    &=
     2 (\delta-1) \lambda \widetilde{D}(t) +
     \frac{N \lambda}{\delta} \sum_{i=1}^N |\widetilde{v}_{i}-v_{i}|^{2}.
    \label{ke_diss}
\end{align}
Next, for $t>\tau$ we use \eqref{CS_delay2} to evaluate the difference
$v_{i}-\widetilde{v}_{i}$,
\begin{equation} \label{vel1_del}
   v_{i}-\widetilde{v}_{i} = \int_{t-\tau}^{t} \tot{}{s} v_{i}(s)\, \d s
    = \frac{\lambda}{N}\sum_{j=1}^N \int_{t-\tau}^{t}
    \widetilde{\psi}_{ij}(s)(\widetilde{v}_{j}(s)-\widetilde{v}_{i}(s))\, \d s.
\end{equation}
Taking the square in \eqref{vel1_del}
and summing over $i$ we have
\begin{align} \label{vel2_del}
   \nonumber
   \sum\limits_{i=1}^N |\widetilde{v}_{i}-v_{i}|^{2}
   &=\frac{\lambda^2}{N^2} \sum\limits_{i=1}^N \left|
   \sum\limits_{j=1}^N \int_{t-\tau}^{t}
   \widetilde{\psi}_{ij}(s)(\widetilde{v}_{j}(s)-\widetilde{v}_{i}(s)) \,\d s \right|^2
   \\ \nonumber
   &\leq \frac{\lambda^2}{N}\sum_{i=1}^N\sum_{j=1}^N \left|
   \int_{t-\tau}^{t} \widetilde{\psi}_{ij}(s)(\widetilde{v}_{j}(s)-\widetilde{v}_{i}(s)) \,\d s \right|^{2}
   \\
   &\leq \frac{\tau \lambda^2}{N}\sum_{i=1}^N\sum_{j=1}^N  \int_{t-\tau}^{t}
   \widetilde{\psi}_{ij}(s) |\widetilde{v}_{j}(s)-\widetilde{v}_{i}(s)|^{2} \,\d s
   = \frac{2\tau \lambda^2}{N} \int_{t-\tau}^{t}\widetilde{D}(s)\,\d s.
\end{align}
The first inequality in \eqref{vel2_del} is
Cauchy-Schwartz for the sum term, i.e. $\left|\sum\limits_{i=1}^N
a_{i}\right|^2 \leq N\sum\limits_{i=1}^N |a_{i}|^2$, and the second
Cauchy-Schwartz inequality for the integral term,
together with the bound $\psi\leq 1$.
Combining \eqref{ke_diss} and \eqref{vel2_del}
directly leads to \eqref{dVest}.
\end{proof}

We now define for $t>\tau$ the functional
\begin{equation} \label{Lyap}
   \L(t) := V(t) + 4\tau\lambda^3 \int_{t-\tau}^{t}\int_{\theta}^{t}  \widetilde{D}(s) \, \d s \, \d\theta,
\end{equation}
where $V=V(t)$ is the quadratic velocity fluctuation \eqref{def:V}.

\begin{lemma} \label{lem:Lyap}
Let the parameters $\lambda, \tau >0$ satisfy
\[ 
  \lambda\tau \leq \frac{1}{2}.
\]
Then along the solutions of \eqref{CS_delay1}--\eqref{CS_delay2}
the functional \eqref{Lyap} satisfies
\[
   \L(t)\leq \L(\tau)\qquad\mbox{for all } t\geq \tau.
\]
\end{lemma}

\begin{proof}
The time derivative of the second term in $\L(t)$ yields
\begin{align*} 
   \tot{}{t}\int_{t-\tau}^{t}\int_{\theta}^{t}
   \widetilde{D}(s) \, \d s \, \d\theta =\tau \widetilde{D}(t)  -
   \int_{t-\tau}^{t} \widetilde{D}(s) \, \d s.
\end{align*}
Combining this with \eqref{dVest} (with the  choice $\delta=1/2$ so that we get rid of the integral term)
 we get
\begin{align*}
   \tot{}{t} \L(t) \leq \left( -1 + 4{\tau^2 \lambda^2}\right) \lambda \widetilde{D}(t).
\end{align*}
We observe that the right-hand side is nonpositive for $\lambda\tau \leq 1/2$ and conclude.
\end{proof}

Consequently, for $\lambda\tau \leq 1/2$ the velocity fluctuation $V=V(t) \leq \L(t)$ is uniformly
bounded from above by $\L(\tau)$ for all $t\geq \tau$.
To cover the interval $[0,\tau]$, where Lemma \ref{lem:Lyap} does not apply,
and establish a connection to the initial datum,
we derive the following estimate on $V=V(t)$ which holds for all $t>0$.

\begin{lemma}\label{lem:estV1}
Along the solutions $(x(t), v(t))$ of the system \eqref{CS_delay1}--\eqref{CS_delay2}
the quadratic fluctuation $V=V(t)$ satisfies
\( \label{estV1}
      \tot{}{t} V(t)  \leq  2\lambda \left(V(t) + \widetilde V(t) \right) \qquad\mbox{for all } t>0.
\)
\end{lemma}

\begin{proof}
With \eqref{CS_delay2} we have, for $t>0$,
\[
   \tot{}{t} V(t) =
   \frac{1}{2} \tot{}{t} \sum_{i=1}^N\sum_{j=1}^N |v_i-v_j|^2 =
      \frac{\lambda}{N} \sum_{i=1}^N\sum_{j=1}^N \left\langle
         v_i-v_j, \sum_{k=1}^N \wpsi_{ik} (\wv_k-\wv_i) - \wpsi_{jk} (\wv_k-\wv_j)
         \right\rangle.
\]
Exchange of the indices $i$ and $j$ gives
\[
   \sum_{i=1}^N \sum_{j=1}^N  \left\langle
         v_i-v_j, \wpsi_{jk} (\wv_k-\wv_j)  \right\rangle
         = - \sum_{i=1}^N \sum_{j=1}^N  \left\langle
         v_i-v_j, \wpsi_{ik} (\wv_k-\wv_i)  \right\rangle,
\]
so that, by the Cauchy-Schwartz inequality and the fact that $\psi\leq 1$,
\[
   \tot{}{t} V(t) &=&
      \frac{2\lambda}{N} \sum_{i=1}^N\sum_{j=1}^N \sum_{k=1}^N \left\langle
         v_i-v_j, \wpsi_{ik} (\wv_k-\wv_i)  \right\rangle \\
   &\leq&
      \frac{\lambda}{N} \sum_{i=1}^N\sum_{j=1}^N \left(
      |v_i-v_j|^2 + \wpsi_{ij}^2 |\wv_i-\wv_j|^2  \right) \\
    &\leq&
      2\lambda \left( V(t) + \widetilde{V}(t) \right).
\]
\end{proof}

Integrating \eqref{estV1} on $(0,\tau)$ we readily have
\[
   V(\tau) \leq (2\lambda\tau + 1) e^{2\lambda\tau} \max_{s\in [-\tau,0]} V(s).
\]
Therefore, the value of $\L(\tau)$ is bounded in terms of the initial datum as
\(   \label{Lbound}
      \L(\tau) \leq  L^0,
\)
with $L^0$ defined in \eqref{def:L0}.

\subsection{Forward-backward estimates}\label{FbEst}

\begin{lemma} \label{lem:D_ineq}
Let the communication rate $\psi=\psi(r)$ satisfy assumption \eqref{ass:psi2}.
Then along the solutions of \eqref{CS_delay1}--\eqref{CS_delay2}, the
quantity $D(t)$ defined by \eqref{def:D} satisfies for any fixed $\eps>0$ the
inequality
\begin{equation} \label{D_ineq}
   |\dot{D}(t)| \leq \left( 2\eps + \frac{\alpha\sqrt{2L^0}}{2} \right) D(t) + \frac{2\lambda^2}{\eps} \widetilde{D}(t) \qquad\mbox{for all } t>0,
\end{equation}
with $\alpha>0$ given in \eqref{ass:psi2}.
\end{lemma}

\begin{proof}
For $t>0$ we differentiate \eqref{def:D} in time,
\begin{align}
  \nonumber
   \tot{}{t} {D}(t) &= \frac{1}{2} \tot{}{t} \sum_{i=1}^N \sum_{j=1}^N \psi_{ij} |v_{i}-v_{j}|^{2} \\
     &= \frac{1}{2} \sum_{i=1}^N \sum_{j=1}^N \psi_{ij}' \left\langle
        \frac{x_{i}-x_{j}}{|x_{i}-x_{j}|},v_{i}-v_{j}\right\rangle |v_{i}-v_{j}|^{2}
        + \sum_{i=1}^N \sum_{j=1}^N \psi_{ij}\langle v_{i}-v_{j},\dot{v}_{i}-\dot{v}_{j}\rangle,
         \label{dD(t)}
\end{align}
where $\psi_{ij}'=\psi'(|x_{i}-x_{j}|)$.
By assumption \eqref{ass:psi2}, $|\psi'(r)|\leq \alpha \psi(r)$ for $r \geq 0$,
we have for the first term of the right-hand side
\begin{align} \nonumber
   \frac{1}{2}\left| \sum_{i=1}^N \sum_{j=1}^N
   \psi_{ij}' \left\langle \frac{x_{i}-x_{j}}{|x_{i}-x_{j}|},v_{i}-v_{j} \right\rangle |v_{i}-v_{j}|^{2} \right|
   &\leq \frac{\alpha}{2} \sum_{i=1}^N \sum_{j=1}^N \psi_{ij}|v_{i}-v_{j}| |v_{i}-v_{j}|^2 \\
   &\leq \frac{\alpha\sqrt{2L^0}}{2}
   \sum_{i=1}^N \sum_{j=1}^N \psi_{ij}|v_{i}-v_{j}|^2
   =  \frac{\alpha\sqrt{2L^0}}{2} D(t),
\label{Est1}
\end{align}
where in the second inequality we used the bound
\[
   |v_i(t) - v_j(t)| \leq \sqrt{2V(t)} \leq \sqrt{2\L(\tau)} \leq \sqrt{2L^0}
\]
provided by Lemma \ref{lem:Lyap} and \eqref{Lbound}.

For the second term of the right-hand side of \eqref{dD(t)}
we apply the symmetrization trick,
\[
   \sum_{i=1}^N \sum_{j=1}^N \psi_{ij}\langle
      v_{i}-v_{j},\dot{v}_{i}-\dot{v}_{j}\rangle
    = 2 \sum_{i=1}^N \sum_{j=1}^N \psi_{ij}\langle
     v_{i}-v_{j},\dot{v}_{i} \rangle,
\]
and estimate using the Cauchy-Schwartz inequality with $\eps>0$,
\begin{align}
   \nonumber
   2 \left| \sum_{i=1}^N \sum_{j=1}^N \psi_{ij}\langle
     v_{i}-v_{j},\dot{v}_{i} \rangle \right|
   & \leq \eps \sum_{i=1}^N \sum_{j=1}^N \psi_{ij}|v_{i}-v_{j}|^2
      +\frac{N}{\eps} \sum_{i=1}^N |\dot{v}_{i}|^2 \\
    \nonumber
   &\leq {\eps} \sum_{i=1}^N \sum_{j=1}^N \psi_{ij}|v_{i}-v_{j}|^2
   + \frac{\lambda^2}{N \eps}\sum_{i=1}^N \left|
        \sum\limits_{j=1}^N \widetilde{\psi}_{ij} (\widetilde{v}_{j}-\widetilde{v}_{i})\right|^{2} \\
    \nonumber
   &\leq {\eps} \sum_{i=1}^N \sum_{j=1}^N \psi_{ij}|v_{i}-v_{j}|^2
     + \frac{\lambda^2}{\eps} \sum_{i=1}^N \sum_{j=1}^N \widetilde{\psi}_{ij}|\widetilde{v}_{j}-\widetilde{v}_{i}|^2 \\
   &\leq
   2 \eps D(t) + \frac{2 \lambda^2}{\eps} \widetilde{D}(t).
\label{Est2}
\end{align}
Combining the estimates \eqref{Est1}--\eqref{Est2} gives \eqref{D_ineq}.
\end{proof}

\begin{lemma} \label{lem:delayEst}
Let $y\in C([-\tau,\infty))$ be a nonnegative function,
differentiable on $(-\tau,0)$ and continuously differentiable on $(0,\infty)$,
satisfying for some constants $C_1 ,C_2, \tau > 0$ the differential inequality
\(   \label{ass:y1}
    |\dot{y}(t)| \leq C_1 y(t) + C_2 y(t-\tau) \qquad \text{for all } t > 0.
\)
Moreover, let
\(  \label{ass:y2}
   M := \max \left\{\sup_{s \in (-\tau,0)} \frac{|\dot{y}(s)|}{y(s)},
                 \frac{|\dot{y}(0^+)|}{y(0)}
              \right\} < \infty,
\)
where $\dot{y}(0^+)$ denotes the right-hand side derivative of $y(t)$ at $t=0$.

Then, if there exists some $\kappa>0$ such that
\(  \label{ass:y3}
    \kappa > \max \left\{ M , C_1+C_2 e^{\kappa \tau} \right\},
\)
then for all $s,t>0$ such that $-\tau < t-s$,
\(  \label{fb}
   e^{-\kappa s} y(t) < y(t-s) < e^{\kappa s} y(t).
 \)
\end{lemma}

\begin{proof}
Due to the assumed continuity of $y(t)$ and $\dot y(t)$
on $(0,\infty)$, \eqref{ass:y2}--\eqref{ass:y3} imply
that there exists $T>0$ such that
\(  \label{proof:y1}
  - \kappa < \frac{\dot{y}(t)}{y(t)} < \kappa \qquad\mbox{for } t\in(-\tau, T).
\)
We claim that \eqref{proof:y1} holds for all $t>-\tau$, i.e., $T=\infty$.
For contradiction, assume that $T<\infty$, then again by continuity we have
\( \label{proof:y2}
   |\dot{y}(T)| = \kappa y(T).
\)
Integrating \eqref{proof:y1} on the time interval $(T-\tau,T)$ yields
\[
   e^{-\kappa\tau} y(T) < y(T-\tau) < e^{\kappa\tau} y(T).
\]
Combining this with \eqref{ass:y1} and \eqref{ass:y3} gives
\[
    |\dot{y}(T)| \leq C_1 y(T) + C_2 y(T-\tau) < \left( C_1 + C_2 e^{\kappa\tau} \right) y(T) < \kappa y(T),
\]
which is a contradiction to \eqref{proof:y2}.
Consequently, \eqref{proof:y1} holds with $T:=\infty$,
and an integration on the interval $(t-s,t)$ implies \eqref{fb}.
\end{proof}

We now apply the result of Lemma \ref{lem:delayEst}
to derive a backward-forward estimate on the quantity
$D=D(t)$ defined in \eqref{def:D}.

\begin{lemma}\label{lem:estD3}
Let the initial datum $(x^0(s),v^0(s))\in C([-\tau,0], \R^{2d})$ be differentiable and such that
the quantity $D=D(t)$ defined in \eqref{def:D} satisfies
\(   \label{ass:D1}
   M := \max \left\{\sup_{s \in (-\tau,0)} \frac{|\dot{D}(s)|}{D(s)},
                 \frac{|\dot{D}(0^+)|}{D(0)}
              \right\} < \infty,
\)
where $\dot{D}(0^+)$ denotes the right-hand side derivative of $D(t)$ at $t=0$
along the solution of the system \eqref{CS_delay1}--\eqref{CS_delay2}.

Moreover, let $\mu>0$ be such that
\(   \label{ass:D2}
   \lambda \mu > \max \left\{ M , 4\lambda e^\frac{\mu\lambda \tau}{2} + \frac{\alpha\sqrt{2L^0}}{2} \right\}.
\)
Then for all $t>0$,
\(  \label{fbV}
   e^{-\mu\lambda\tau} D(t) < D(t-\tau) < e^{\mu\lambda\tau} D(t).
 \)
\end{lemma}

\begin{proof}
Choosing $\eps:= \lambda e^\frac{\mu\lambda\tau}{2}$ in formula \eqref{D_ineq} of Lemma \ref{lem:D_ineq} gives
\[
   |\dot D| \leq C_1 D + C_2 \widetilde{D} \qquad\mbox{for all } t>0
\]
with
\[
   C_1 := 2\lambda e^\frac{\mu\lambda\tau}{2} + \frac{\alpha\sqrt{2L^0}}{2}, \qquad
   C_2 := 2\lambda e^\frac{-\mu\lambda\tau}{2}.
\]
The statement follows then directly from Lemma \ref{lem:delayEst} with $y:=D$, $\kappa:=\lambda\mu$ and
noting that $C_1+C_2 e^{\kappa \tau} = 4\lambda e^\frac{\mu\lambda\tau}{2} + \frac{\alpha\sqrt{2L^0}}{2}$.
\end{proof}

\begin{remark}\label{rem:constantIC}
Let us note that the rather technical assumptions \eqref{ass:D1}, \eqref{ass:D2} of Lemma \ref{lem:estD3}
simplify greatly in the case that the initial datum $(x^0,v^0)$ 
is constant on $[-\tau, 0]$.
Then \eqref{ass:D1} reduces to
\[
   M := \frac{|\dot{D}(0^+)|}{D(0)} < \infty.
\]
Applying Lemma \ref{lem:D_ineq} with $\eps:=\lambda$,  noting that $D(0)=\widetilde{D}(0)\neq 0$, gives
\[
   \frac{|\dot{D}(0^+)|}{D(0)} \leq \left( 4\lambda + \frac{\alpha\sqrt{2L^0}}{2} \right).
\]
Consequently, \eqref{ass:D2} is verified when
\[
    \lambda \mu > 
       4\lambda e^\frac{\mu\lambda \tau}{2} + \frac{\alpha\sqrt{2L^0}}{2}.
\]
A simple calculation (checking for the positivity of the global maximum of the function
$f(x) = x - 4\lambda e^\frac{x \tau}{2} - \frac{\alpha\sqrt{2L^0}}{2}$)
reveals that the above condition is satisfiable if and only if
\(  \label{cond:simplified}
   \frac{\alpha\sqrt{2L^0}}{2\lambda} < \frac{2}{\lambda\tau} \left(\ln\frac{1}{2\lambda\tau} - 1\right),
\)
which implies the necessary condition $\lambda\tau < \frac{1}{2e}$.
The function $f(x)=\frac{1}{x}\ln\left(\frac{1}{2x}-1\right)$ is on the interval $(0,\frac{1}{2e})$
monotonically decreasing from $+\infty$ to zero,
while $L^0$ takes the form
\(    \label{L0const}
      L^0 = (2\lambda\tau + 1) e^{2\lambda\tau} V(0) + 2\lambda^3\tau^3 D(0).
\)
Consequently, for any $\lambda>0$ and any values of $V(0)$, $D(0)$ there exists $\tau_c>0$
such that \eqref{cond:simplified} is true for all $0 < \tau < \tau_c$.
\end{remark}

\subsection{Decay of the velocity fluctuations and flocking}\label{subsec:EstV}

In order to bound $D=D(t)$ from below by the quadratic velocity fluctuation $V=V(t)$,
we need to introduce the minimum interparticle interaction $\varphi = \varphi(t)$,
\(  \label{def:varphi}
   \varphi(t):=\min_{i,j=1,\cdots,N} \psi(|x_{i}(t)-x_{j}(t)|),
\)
and the position diameter
\(  \label{def:X}
   d_X(t):=\max_{i,j=1,\cdots,N} |x_i(t)-x_j(t)|.
\)
We then have the following estimate:

\begin{lemma} \label{lem:EstPhi}
Let the parameters $\tau, \lambda >0$ satisfy
\[
  \lambda\tau \leq \frac{1}{2}.
\]
Then along the solutions of \eqref{CS_delay1}--\eqref{CS_delay2}
we have
\(  \label{EstPhi}
   \varphi(t) \geq \psi\left( d_X(0) + \sqrt{2L^0}t \right) \qquad\mbox{for all } t>\tau.
\)
\end{lemma}

\begin{proof}
Since, by assumption, $\psi=\psi(r)$ is a nonincreasing function, we have
\(   \label{varphi_est}
   \varphi(t) = \min_{i,j=1,\cdots,N} \psi(|x_{i}(t)-x_{j}(t)|) = \psi(d_X(t)),
\)
with $d_X=d_X(t)$ defined in \eqref{def:X}.
Moreover, we have for all $i,j = 1,\cdots,N$,
\[
    \tot{}{t} |x_i - x_j|^2 \leq 2 |x_i-x_j| |v_i-v_j|,
\]
and Lemma \ref{lem:Lyap} together with \eqref{Lbound}
gives
\[
   |v_i(t) - v_j(t)|^2 \leq 2V(t) \leq 2\L(\tau) \leq 2L^0. 
\]
Consequently,
\[
   \tot{}{t} |x_i - x_j|^2 \leq 2\sqrt{2L^0} |x_i-x_j|,
\]
and integrating in time and taking the maximum over all $i,j = 1,\cdots,N$ yields
\[
   d_X(t) \leq d_X(0) + \sqrt{2L^0} t,
\]
which combined with \eqref{varphi_est} directly implies \eqref{EstPhi}.
\end{proof}

We are now in position to provide a proof of Theorem \ref{thm:main}.
\begin{proof}
Given that the assumptions of Lemma \ref{lem:estD3} are satisfied,
we apply \eqref{fbV} to \eqref{dVest} to obtain, for $t>\tau$,
\[
   \tot{}{t} V(t) &\leq& 2(\delta - 1)\lambda e^{-\mu\lambda\tau} D(t)
      + \frac{2\tau \lambda^3}{\delta} D(t) \int_{0}^{\tau} e^{\mu\lambda\tau} \,\d s \\
   &\leq&
   2\lambda e^{-\mu\lambda\tau} \left( \delta - 1 + \frac{\tau^2 \lambda^2}{\delta} e^{2\mu\lambda\tau} \right) D(t).
\]
Optimizing in $\delta>0$ gives $\delta:=\lambda\tau e^{\mu\lambda\tau}$ and
\[
   \tot{}{t} V(t) \leq  2\lambda e^{-\mu\lambda\tau} \left( 2\lambda\tau e^{\mu\lambda\tau} - 1 \right) D(t).
\]
By the definition \eqref{def:varphi} of the minimal interaction $\varphi=\varphi(t)$ we have the estimate
\[
   D(t) = \frac{1}{2} \sum_{i=1}^N \sum_{j=1}^N \psi_{ij} |v_i - v_j|^2
   \geq \frac{1}{2} \varphi(t) \sum_{i=1}^N \sum_{j=1}^N |v_i - v_j|^2
   = \varphi(t) V(t).
\]
Consequently, under the assumption
\(  \label{ass:lt}
   2\lambda\tau e^{\mu\lambda\tau} < 1,
\)
whose validity will be discussed below, we have
\[
   \tot{}{t} V(t) \leq  2\lambda e^{-\mu\lambda\tau} \left( 2\lambda\tau e^{\mu\lambda\tau} - 1 \right) \varphi(t) V(t).
\]
Denoting $\omega:=-2\lambda e^{-\mu\lambda\tau} \left( 2\lambda\tau e^{\mu\lambda\tau} - 1 \right)>0$, integrating in time
and applying \eqref{EstPhi} yields
\(   \label{flocking}
   V(t) \leq V(\tau) \exp\left( -\omega \int_\tau^t \varphi(s) \d s \right)
      \leq V(\tau) \exp\left(-\omega \int_\tau^t \psi\left(d_X(\tau) + \sqrt{2L^0}s \right) \d s \right).
\)
Consequently, if $\int^\infty \varphi(s) \d s=\infty$,
we have the asymptotic convergence of the velocity fluctuation to zero,
$\lim_{t\to\infty} V(t) = 0$.

By assumption \ref{ass:psi1}, namely that $\psi(r) \geq Cr^{-1+\gamma}$ for all $r>R$,
we have, asymptotically for large $t>0$,
\[
    \int^t \psi\left(d_X(\tau) + \sqrt{2L^0}s \right) \d s \gtrsim t^\gamma.
\]
Consequently, from \eqref{flocking},
\[
   V(t) \lesssim \exp\left(-\omega t^\gamma \right).
\]
A slight modification of the proof of Lemma \eqref{lem:EstPhi} gives
\[
   d_X(t) \leq d_X(\tau) + \int_\tau^t \sqrt{V(s)} \d s \lesssim d_X(\tau) + \int_\tau^t  \exp\left(-\omega s^\gamma/2 \right) \d s \qquad\mbox{for } t\geq \tau.
\]
The integral on the right-hand side is uniformly bounded,
implying the uniform boundedness of the position diameter $d_X(t) \leq \bar d_X <+\infty$
for all $t>0$. This in turn implies $\varphi(t) \geq \varphi(d_X(t)) \geq \varphi(\bar d_X)$
and from \eqref{flocking} we obtain the exponential decay of the velocity fluctuations
\[
      V(t) \leq V(\tau) \exp\left(-\omega \varphi(\bar d_X) (t-\tau) \right) \qquad\mbox{for all } t>\tau.
\]

It remains to discuss the validity of the assumptions
\eqref{ass:D2} and, resp., \eqref{ass:lt},
which we reprint here for the convenience of the reader:
\[
   \lambda \mu &>& \max \left\{ M , 4\lambda e^\frac{\mu\lambda \tau}{2} + \frac{\alpha\sqrt{2L^0}}{2} \right\},\\
   2\lambda\tau e^{\mu\lambda\tau} &<& 1,
\]
with $L^0$ given by \eqref{def:L0} and $M$ given by \eqref{ass:D1},
\[
   M := \max \left\{\sup_{s \in (-\tau,0)} \frac{|\dot{D}(s)|}{D(s)},
                 \frac{|\dot{D}(0^+)|}{D(0)}
              \right\} < \infty.
\]
A simple calculation reveals that a necessary condition for the verifiability of \eqref{ass:D2}
is $\lambda\tau < \frac{1}{2e}$.
We shall now show that, for a given initial datum $(x^0,v^0)$ and fixed values of $\lambda$ and $M$
there exists a critical delay $\tau_c$ such that if $\tau<\tau_c$, there exists $\mu>0$ such that
both assumptions are verified.
We denote
\[
   K :=  \frac{1}{\lambda} \max \left\{ \frac{|\dot{D}(0^+)|}{D(0)}, 4\lambda + \frac{\alpha\sqrt{2V(0)}}{2} \right\} < +\infty.
\]
Noting that \eqref{ass:lt} is equivalent to
\[
   \mu < \frac{1}{\lambda\tau} \ln\frac{1}{2\lambda\tau},
\]
let us choose $\tau_1>0$ such that
\(  \label{tau1:choice}
    K < \frac{1}{\lambda \tau_1} \ln\frac{1}{2\lambda\tau_1}.
\)
Note that the right-hand side above tends to $+\infty$ as $\tau_1\to 0+$,
therefore, such $\tau_1>0$ always exists.
Then, choose
\(  \label{mu:choice}
   \mu := \frac{1}{\lambda\tau_1} \ln\frac{1}{2\lambda\tau_1}.
\)
Note that, since the function $\frac{1}{x}\ln\frac{1}{2x}$ is decreasing for $x\in(0,\frac{1}{2e})$,
choosing $\tau<\tau_1$ guarantees that \eqref{mu:choice}, and thus \eqref{ass:lt}, is verified.
Then, for the above value of $\mu$, choose $\tau_2>0$ such that \eqref{ass:D2} is verified.
Again, this is possible since the right-hand side of \eqref{ass:D2} is equal to $\lambda K$ for $\tau=0$
(recall that $L^0=V(0)$ for $\tau=0$), it is continuous in $\tau$ and $\mu>K$.
Obviously, choosing $\tau_c := \min\{\tau_1, \tau_2\} >0$, both the conditions \eqref{ass:D2}, \eqref{ass:lt}
are satisfied for any $0 < \tau < \tau_c$, with $\mu$ given by \eqref{mu:choice}.

A more optimal value of the critical delay $\tau_c$ can be obtained if the initial datum $(x^0, v^0)$
is constant.
Then \eqref{ass:D1} reduces to
\[
   M := \frac{|\dot{D}(0^+)|}{D(0)} < \infty.
\]
Applying Lemma \ref{lem:D_ineq} with $\eps:=\lambda$,  noting that $D(0)=\widetilde{D}(0)\neq 0$, gives
\[
   \frac{|\dot{D}(0^+)|}{D(0)} \leq \left( 4\lambda + \frac{\alpha\sqrt{2L^0}}{2} \right).
\]
Consequently, \eqref{ass:D2} becomes
\(  \label{Ass:N3}
   \lambda \mu > 4\lambda e^\frac{\mu\lambda \tau}{2} + \frac{\alpha\sqrt{2L^0}}{2}.
\)
A simple analysis (c.f. \eqref{cond:simplified} and \eqref{L0const} in Remark \ref{rem:constantIC})
reveals that the maximum value of $\tau>0$ for which the above
is satisfiable with some $\mu>0$ is the unique solution of the equation
\[
    \frac{2}{\lambda\tau} \left(\ln\frac{1}{2\lambda\tau} - 1\right) = \frac{\alpha\sqrt{2L^0}}{2\lambda}.
\]
Let us denote this solution $\tau_1$ and choose
\[
   \mu:= \frac{\alpha\sqrt{2L^0}}{2\lambda} + \frac{2}{\lambda\tau_1}.
\]
Then \eqref{Ass:N3} is verified for each $0 < \tau < \tau_1$.
It then remains to choose $\tau_2>0$ such that  \eqref{ass:lt} is satisfied
and $\tau_c := \min\{\tau_1, \tau_2\}$.

\end{proof}

\begin{remark} \label{rem:N_depend}
A point worth mentioning is the asymptotic dependence of the approximation of
the critical value $\tau_c$ calculated above on the number of
agents $N$ for $N\gg 1$.
Since the expression \eqref{def:L0} for $L^0$ contains $N^2$
summation terms, we have $L^0 = \mathcal{O}(N^2)$,
while $M$, given by \eqref{ass:D1}, is of order $\mathcal{O}(1)$.
Consequently, in \eqref{tau1:choice} we may choose
$\tau_1 = \mathcal{O}(N^{-1})$.
This gives $\mu = \mathcal{O}(N\ln N)$ in \eqref{mu:choice}
and, in turn, $\tau_2 = \mathcal{O}(N^{-1})$ in \eqref{ass:D2}.
We conclude that $\tau_c = \mathcal{O}(N^{-1})$.

Let us note that the results of \cite{ErHaSu}, in particular, Theorem 3.1
and the numerical simulations in Section 4 suggest that,
as long as the right-hand side in \eqref{CS_delay2} is rescaled by $1/N$,
the critical delay does not depend on the number of agents $N$.
However, our method depends on relating the value of $\tau_c$
to the (unscaled) quadratic fluctuation \eqref{def:V} of the initial datum.
Therefore, it is not clear if an analytical estimate for the critical delay uniform in $N$
can be obtained using an appropriate modification of this method.
It is possible that a more elaborate treatment of the evolution of the velocity diameter
$d_V(t)$ might give a control which is uniform in $N$ and depends only on initial diameters and
the weight function $\psi(\cdot)$. We leave this question for a future work.
\end{remark}

\section{A remark about the mean-field description}\label{sec:meanfield}
While derivation of mean-field limits for Cucker-Smale type systems
without delay and their analysis is a well understood topic, see,
e.g., \cite{CCH2, CCR}, the introduction of delay terms may lead
to an ill-posed problem. Let us formally demonstrate this fact for
our system \eqref{CS_delay1}--\eqref{CS_delay2}.

We fix the initial particle density $f_0\in \mathcal{P}(\R^{2d})$,
where $\mathcal{P}(\R^{2d})$ denotes the set of probability measures
on the space $\R^{2d}$, and construct the time-dependent density
$f_t\in \mathcal{P}(\R^{2d})$ for $t>0$ by the push-forward \(
\label{ft}
   f_t := Z_t[f_{t-\tau}]\# f_0,
\) where $Z_t[g_t] := (Z_t^x, Z_t^v)[g_t]: \R^{2d} \to \R^{2d}$ is
the characteristic mapping corresponding to the system
\eqref{CS_delay1}--\eqref{CS_delay2}, i.e., \(  \label{char1}
   \tot{Z_t^x}{t} &=& Z_t^v,\\
   \label{char2}
   \tot{Z_t^v}{t} &=& F[g_t] (Z_{t-\tau}),
\) with $F[g_t]: \R^{2d}\to \R^{2d}$, \(  \label{Fgt}
   F[g_t](x,v) := \int_{\R^{2d}} \psi(|x-y|) (w-v) \d g_t(y,w),
\) for any time-dependent probability measure $g_t$. We shall write
$Z_t$ instead of $Z_t[f_{t-\tau}]$ in the sequel for the easy of
notation. Let us recall the change-of-variables formula
\[
   \int_{\R^{2d}} \eta\, \d(Z_t\#f_0) = \int_{\R^{2d}} \eta(Z_t) \d f_t
\]
for any function $\eta:\R^{2d}\to\R$ integrable with respect to the
push-forward measure $Z_t\# f_0$. Let us fix a smooth, compactly
supported test function $\xi\in C^\infty_c(\R^{2d})$ and calculate,
using \eqref{ft} and the above change-of-variables formula,
\[
   \tot{}{t} \int_{\R^{2d}} \xi \d f_t &=&
     \tot{}{t} \int_{\R^{2d}} \xi(Z_t) \, \d f_0 \\
     &=&
     \int_{\R^{2d}} \nabla_x \xi(Z_t) \cdot Z_t^v
        + \nabla_v \xi(Z_t) \cdot F[f_{t-\tau}] (Z_{t-\tau}) \, \d f_0 \,.
\]
Using again the change-of-variables formula, the right-hand side can
be written in the form
\[
     \int_{\R^{2d}} \nabla_x \xi \cdot v + \nabla_v \xi
     \cdot F[f_{t-\tau}] (Z_{t-\tau}(Z_t^{-1})) \d f_t \,.
\]
Consequently, if the inverse mapping $Z_t^{-1}$ was well defined,
then the formal mean-field limit for the system
\eqref{CS_delay1}--\eqref{CS_delay2} as $N\to\infty$ could be
written in the form \(  \label{kin0}
    \partial_t f_t + v\cdot\nabla_x f_t + \nabla_v\cdot (F[f_{t-\tau}]
    (Z_{t-\tau}(Z_t^{-1})) f_t) = 0,
\) subject to the initial datum $f_0\in\mathcal{P}(\R^{2d})$, with
$Z_t=Z_t[f_{t-\tau}]$ the solution of the characteristic system
\eqref{char1}--\eqref{char2}. Obviously, the composed mapping
$Z_{t-\tau}(Z_t^{-1})$ is supposed to act as a ``time machine'',
mapping the phase-space coordinates of a given characteristic at
time $t$ to its coordinates at time $t-\tau$. However, in general
the characteristics may overlap, so that the inverse mapping
$Z_t^{-1}$ is not defined. This can be seen for instance in the
trivial case of the ``noninteracting'' particles
\[
   \dot v_i(t) = -v_i(t-1)
\]
subject to constant initial datum on $(-1,0)$. For this system all
trajectories (i.e. characteristics) overlap at $t=1$ at zero.
 Let us also remark that even if the characteristic mapping
 $Z_t$ was surjective and thus invertible
 (say, for short enough time intervals),
 the ``practical'' usability of the kinetic equation \eqref{kin0} for numerical simulations
 would be strongly limited by the necessity to track the characteristics back in time
 through $Z_{t-\tau}(Z_t^{-1})$.

 Finally, let us remark that the situation strongly depends on the particular structure
 of the delay system.
 In particular, let us consider the modification of \eqref{CS_delay1}--\eqref{CS_delay2}
where the terms $x_i$ and $v_i$ in the equation for $\dot v_i$ are
evaluated in the present time $t$, i.e.,
\[
      \dot{x}_{i} &=& v_{i}  \\ 
      \dot{v}_{i} &=& \frac{\lambda}{N}\sum\limits_{j=1}^N
      \psi(|{x}_{i}-\widetilde{x}_{j}|)
      (\widetilde{v}_{j}-{v}_{i}). 
\]
For this case, one can easily show that the formal mean-field
limit as $N\to\infty$ is given by the kinetic equation
 \(  \label{kin2}
    \partial_t f_t + v\cdot\nabla_x f + \nabla_v\cdot (F[t_{t-\tau}] f_t) = 0,
 \)
 with $F[t_{t-\tau}]$ given by \eqref{Fgt}.
Obviously, there is no need for tracking of characteristics back
in time in this case, and standard analytical and numerical methods
apply to \eqref{kin2}.

\section*{Acknowledgment}
JH acknowledges the support of the KAUST baseline funds.

\quad (Jan Haskovec)
\\
\textsc{Computer, Electrical and Mathematical Sciences \&
Engineering}
\\
\textsc{King Abdullah University of Science and Technology, 23955
Thuwal, KSA}

\quad E-mail address: \textbf{jan.haskovec@kaust.edu.sa}
\\ \\

\quad (Ioannis Markou)
\\
\textsc{Institute of Applied and Computational Mathematics
(IACM-FORTH)}
\\
\textsc{N. Plastira 100, Vassilika Vouton GR - 700 13, Heraklion,
Crete, Greece}

\quad E-mail address: \textbf{ioamarkou@iacm.forth.gr}


\begin{thebibliography}{99}




\bibitem{CaFoRoTo} \textsc{Carrilo., J.A., Fornasier., M., Rosado., J., and Toscani., G.}
\emph{Asymptotic flocking dynamics for the kinetic Cucker-Smale
model.} SIAM J. Math. Anal. \textbf{42}, 1 (2010), 218--236.

\bibitem{CCH2}
\newblock {J. A. Carrillo, Y.-P. Choi and M. Hauray},
\newblock {The derivation of swarming models: Mean-field limit and Wasserstein distances},
 \newblock \emph{Collective Dynamics from Bacteria to Crowds: An Excursion Through Modeling, Analysis and Simulation}
 (eds. A. Muntean and F. Toschi), Springer Series: CISM International Centre for Mechanical Sciences, 533 (2014), 1--46.

\bibitem{CCR}
\newblock {J. Ca\~nizo, J. Carrillo and J. Rosado},
\newblock {A well-posedness theory in measures for some kinetic models of collective motion},
\newblock  \emph{Math. Mod. Meth. Appl. Sci.}, \textbf{21}, 3 (2011), 515--539.

\bibitem{ChHaLi} \textsc{Choi, Y.-P., Ha S.-Y., and Li., Z.}
\emph{Emergent Dynamics of the Cucker--Smale Flocking Model and Its
Variants.} In: Bellomo N., Degond P., Tadmor E. (eds) Active
Particles, Volume 1. Modeling and Simulation in Science, Engineering
and Technology. Birkh\"auser (2017).

\bibitem{ChHa1} \textsc{Choi, Y.-P., and Haskovec, J.} \emph{Cucker-Smale model with
normalized communication weights and time delay.}
Kinet. Relat. Models \textbf{10}, 4 (2017), 1011--1033.

\bibitem{ChHa2} \textsc{Choi, Y.-P., and Haskovec, J.} \emph{Hydrodynamic Cucker-Smale model
with normalized communication weights and time delay.}
SIAM J. Math. Anal. \textbf{51}, 3 (2019), 2660--2685.

\bibitem{ChLi} \textsc{Choi, Y.-P., and Li, Z.} \emph{Emergent behavior of Cucker-Smale
flocking particles with heterogeneous time delays.} Appl. Math.
Lett. 86, (2018), 49--56.


\bibitem{CuSm1} \textsc{Cucker, F., and Smale, S.} \emph{Emergent behavior in
flocks.} IEEE Trans. Automat. Control \textbf{52}, 5 (2007),
852-862.

\bibitem{CuSm2} \textsc{Cucker, F., and Smale, S.} \emph{On the mathematics of
emergence.} Japan J. Math. \textbf{2}, 1 (2007), 197--227.

\bibitem{DHK19}
\textsc{Dong, J.-G., Ha, S.-Y., and Kim, D.}
\emph{Interplay of time-delay and velocity alignment in the Cucker-Smale model on a general digraph.}
DCDS-B \textbf{24} (2019), 5569--5596.

\bibitem{ErHaSu} \textsc{Erban, R., Haskovec, J., and Sun, Y.}
\emph{A Cucker-Smale model with noise and delay.}
SIAM J. Appl. Math. \textbf{76}, 4 (2016), 1535--1557.

\bibitem{Gyori-Ladas}
\textsc{Gyori I., and Ladas G.} \emph{Oscillation Theory of Delay Differential Equations with Applications.}
Oxford Science Publications, Clarendon Press, Oxford, 1991.

\bibitem{Ha1}
\textsc{Ha, S.-Y., Kim, J., Park, J., and Zhang X.} \emph{Complete
Cluster Predictability of the Cucker-Smale Flocking Model on the
Real Line.} Arch. Rational Mech. Anal. 231 (2019), 319--365.

\bibitem{Kuramoto}
\textsc{Ha, S.-Y., Lattanzio, C., Rubino, B., and Slemrod, M.}
\emph{Flocking and synchronization of particle models.}
Quarterly of Applied Mathematics \textbf{69} (2011), 91--103.


\bibitem{HaLiu}
\textsc{Ha, S.-Y., and Liu J.-G.}
\emph{A simple proof of the Cucker-Smale flocking dynamics and mean-field limit.}
Communications in mathematical sciences \textbf{7}, 2 (2009), 297--325.

\bibitem{HaTa} \textsc{Ha, S.-Y., and Tadmor, E.}
\emph{From particle to kinetic and hydrodynamic descriptions of flocking.}  Kinet. Relat. Models
\textbf{1}, 3 (2008), 415--435.

\bibitem{Kalise}
\textsc{Kalise, D., Peszek, J., Peters, A., and Choi, Y.-P.} \emph{A
collisionless singular Cucker-Smale model with decentralized
formation control.}  DOI: 10.13140/RG.2.2.31480.96005 (2018).



\bibitem{LiWu} \textsc{Liu, Y., and Wu, J.} \emph{Flocking and asymptotic velocity of
the Cucker-Smale model with processing delay.} J. Math. Anal. Appl.
\textbf{415}, 1 (2014), 53--61.

\bibitem{Ma} \textsc{Markou, I.} \emph{Collision avoiding in the singular Cucker-Smale model with
nonlinear velocity couplings.} Discrete Contin. Dyn. Syst.
\textbf{38}, 10 (2018), 5245--5260.


\bibitem{Naldi-Pareschi-Toscani}
\textsc{Naldi, G., Pareschi, L. and Toscani, G. (eds.)} \emph{
Mathematical Modeling of Collective behaviour in Socio-Economic and
Life Sciences}, Series: Modelling and Simulation in Science and
Technology, Birkh\"auser, 2010.


\bibitem{Pareschi-Toscani-survey}
{\sc Pareschi, L., and Toscani, G.}, {\em Interacting Multiagent
Systems: Kinetic equations and Monte Carlo methods}, Oxford
University Press, 2014.

\bibitem{PiTr}
\textsc{Pignotti, C., and Trelat, E.} \emph{Convergence to consensus
of the general finite-dimensional Cucker-Smale model with
time-varying delays.} Commun. Math. Sci. \textbf{16}, 8 (2018),
2053--2076.

\bibitem{PiVa18}
\textsc{Pignotti, C., and Reche Vallejo, I.}
\emph{Flocking estimates for the Cucker-Smale model with time lag and hierarchical leadership.}
J. Math. Anal. Appl. 464 (2018), 1313--1332.

\bibitem{PiVa}
\textsc{Pignotti, C., and Reche Vallejo, I.} \emph{Asymptotic Analysis
of a Cucker-Smale System with Leadership and Distributed Delay.}
In: Alabau-Boussouira F., Ancona F., Porretta A., Sinestrari C. (eds) Trends
in Control Theory and Partial Differential Equations. Springer INdAM
Series, vol. 32 (2019).

\bibitem{Smith}
{\sc H. Smith}, {\em An Introduction to Delay Differential
Equations with Applications to the Life Sciences},
Springer New York Dordrecht Heidelberg London, 2011.


\bibitem{Vicsek-survey}
T. Vicsek, A. Zafeiris: \emph{Collective motion.} Phys. Rep.
\textbf{517}, 3-4 (2012), 71--140.

\end{thebibliography}
\end{document}